\documentclass[10pt]{amsart}

\usepackage{eucal,url,amssymb,stmaryrd,booktabs,
enumerate,amscd,paralist}
\usepackage[pagebackref,colorlinks=true]{hyperref}  

\usepackage{amsfonts}
\usepackage{amsmath,amsthm,amssymb,amscd,enumerate,eucal,url,stmaryrd}

\usepackage{eucal,url,amssymb,stmaryrd,
enumerate,amscd,
}

\numberwithin{equation}{section}

\newtheorem{thrm}{Theorem}[section]

\newtheorem*{thrm*}{Theorem}
\newtheorem{lemma}[thrm]{Lemma}

\newtheorem{rmrk}[thrm]{Remark}
\newtheorem{conv}[thrm]{Convention}

\newcommand{\bG}{\boldsymbol {G}}

\newcommand{\Rn}{\mathbb{R}^n}

\newcommand{\Hn}{\mathbb{H}^n}

\newcommand{\QH}{\boldsymbol {G}}

\newcommand{\lap}{\mathcal{L}}

\newcommand{\abs}[1]{\lvert #1 \rvert}

\newcommand{\e}{\textbf {e}}

\newcommand{\domG}{\overset {o}{\mathcal{D}}\,^{1,2}(\bG)}

\newcommand{\dxa}[1]{{\partial_{x_{\alpha}}} {#1}}
\newcommand{\dta}[1]{{\partial_{t_{\alpha}}} {#1}}
\newcommand{\dya}[1]{{\partial_{y_{\alpha}}} {#1}}
\newcommand{\dza}[1]{{\partial_{z_{\alpha}}} {#1}}

\newcommand{\dx}[1]{{\partial_x} {#1}}
\newcommand{\dy}[1]{{\partial_y} {#1}}
\newcommand{\dz}[1]{{\partial_z} {#1}}

\def\sideremark#1{\ifvmode\leavevmode\fi\vadjust{\vbox to0pt{\vss
 \hbox to 0pt{\hskip\hsize\hskip1em
 \vbox{\hsize2.5cm\tiny\raggedright\pretolerance10000
 \noindent #1\hfill}\hss}\vbox to8pt{\vfil}\vss}}}%

\renewcommand{\sideremark}[1]{}

\begin{document}
\begin{abstract}
We determine the best (optimal) constant in the $L^2$ Folland-Stein inequality on the
quaternionic Heisenberg group and the non-negative functions for which equality holds.
\end{abstract}

\title[The optimal constant in the $L^2$  Folland-Stein inequality ]
{The optimal constant in the $L^2$  Folland-Stein inequality on the quaternionic Heisenberg group}
\date{\today}
\thanks{}

\author{Stefan Ivanov}
\address[Stefan Ivanov]{University of Sofia, Faculty of Mathematics and Informatics,
blvd. James Bourchier 5, 1164, Sofia, Bulgaria} \email{ivanovsp@fmi.uni-sofia.bg}

\author{Ivan Minchev}
\address[Ivan Minchev]{University of Sofia\\
Sofia, Bulgaria\\
and  Mathematik und Informatik\\
Philipps-Universit\"at Marburg\\
Hans-Meerwein-Str. / Campus Lahnberge
35032 Marburg, Germany} \email{minchevim@yahoo.com} \email{minchev@fmi.uni-sofia.bg}

\author{Dimiter Vassilev}
\address[Dimiter Vassilev]{
Department of Mathematics and Statistics\\ University of New Mexico\\
Albuquerque, New Mexico, 87131-0001
} \email{vassilev@math.unm.edu} \maketitle


\setcounter{tocdepth}{2}

\tableofcontents

\section{Introduction}

The goal of this note is to determine the best (optimal) constant in the $L^2$ Folland-Stein inequality on the
quaternionic Heisenberg group and  the non-negative extremal functions, i.e., the functions for which equality holds. Alternatively, this is equivalent to finding the Yamabe
constant of the standard quaternionic contact structure of the sphere. The proof relies on the  realization of Branson, Fontana and Morpurgo
\cite{BFM}, recently remarkably used also by Frank and Lieb \cite{FL}, that the old idea of Szeg\"o  \cite{Sz}, see also Hersch \cite{He}, can
be used to find the sharp form of (logarithmic) Hardy-Littlewood-Sobolev type inequalities on the Heisenberg group.

The conformal nature of the problem we consider is key to its solution. The analysis is purely
analytical. In this respect, even though the { quaternionic contact (qc)} Yamabe functional is involved, the {qc} scalar curvature is used in the proof without much geometric meaning. Rather, it is the conformal sub-laplacian that plays a central role and the {qc} scalar curvature appears as a constant determined by the Cayley transform and the left-invariant sub-laplacian on the quaternionic Heisenberg group. In
this respect, this method does not give all solutions of the {qc} Yamabe equation on the {quaternionic contact} sphere. The complete
solution of the latter problem requires some additional very non-trivial argument and it is at this place where
the geometric nature of the problem becomes even more important. In the CR setting, the solution of the CR
Yamabe problem was achieved with the help of an ingenious divergence formula \cite{JL3}. The other known
sub-Riemannian case is that of the {qc} Yamabe equation on the seven dimensional standard quaternionic contact sphere \cite{IMV1}.  Another relevant result appeared earlier \cite{GV}, where the {sub-Riemannian} Yamabe equation was solved in the unifying setting of groups of Iwasawa type under an additional assumption of partial symmetry of the solution. This result can be used   at the final stage of all known proofs after such symmetry has been shown to exist. We recall that the groups of Iwasawa type comprise of the complex (="usual"), quaternion and octonian Heisenberg groups, which are defined by \eqref{e:heisenberg group} replacing, correspondingly, the quaternions $\mathbb{H}$ with the complex numbers $\mathbb{C}$, the quaternions $\mathbb{H}$, and the octonians $\mathbb{O}$.

Given a {compact} quaternionic contact manifold $M$ of real dimension $4n+3$ with {an $\mathbb{R}^3$-valued contact form $\eta=\{\eta_1,\eta_2,\eta_3\}$, i.e. a codimension three horizontal distribution $H$ determined as the kernel of $\eta$ such that $d\eta_{|_H}$ are the fundamental two forms of a quaternionic hermitian structure $({\bf g},I_1,I_2,I_3)$ on $H$}, $(d\eta_{s})_{|_H}=2{\bf g}(I_s.,.)=2\omega_s, s=1,2,3$, a natural question is to determine the {qc} Yamabe constant of {the conformal class $[\eta]$ of $\eta$} defined as the infimum
\begin{equation}\label{e:Yamabe constant problem}
\lambda(M, [\eta])\ =\ \inf \{ \Upsilon (u) :\ \int_M u^{2^*}\, Vol_\eta \ =\ 1, \ u>0\},
\end{equation}
where $Vol_\eta=\eta_1\wedge\eta_2\wedge\eta_3\wedge (\omega_1)^{2n}$ denotes the volume form determined by $\eta$. The {qc} Yamabe functional of the conformal class of $\eta$ is defined by
\begin{equation*}
\Upsilon (u)\ =\ \int_M\Bigl(4\frac {Q+2}{Q-2}\ \lvert \nabla u \rvert^2\ +\ S\, u^2\Bigr)
Vol_\eta,\qquad \int_M u^{2^*}\, Vol_\eta \ =\ 1, \ u>0,
\end{equation*}
denoting by $\nabla$ the  Biquard connection \cite{Biq1} of $\eta$, and $S$ standing for the qc scalar curvature of $(M,\, \eta)$. This is the so called \emph{qc Yamabe constant problem}.
In this paper we shall find $\lambda(S^{4n+3}, [\tilde\eta])$, where $\tilde\eta$ is the standard qc form on the unit sphere $S^{4n+3}$, see \eqref{e:stand cont form on S}. The question is of course related to the solvability of the {qc} Yamabe equation
\begin{equation}\label{e:conf Yamabe}
\mathcal{L} u\ \equiv\ 4\frac {Q+2}{Q-2}\ \triangle u -\ S\, u \ =\ - \  \overline{S}\,u^{2^*-1},
\end{equation}
where $\triangle $ is the horizontal sub-Laplacian, $\triangle u\ =\ tr^g(\nabla du)$, $S$ and
$\overline{S}$ are the qc scalar curvatures correspondingly of $(M,\, \eta)$ and $(M, \, \bar\eta)$,
 $\bar\eta=u^{4/(Q-2)}\eta$, and $2^* = \frac {2Q%
}{Q-2}.$ Here, and throughout the paper $Q=4n+6$ is the homogeneous dimension. The natural question is to find
all solutions of the {qc} Yamabe equation. This is the so called \emph{qc Yamabe problem}, which is equivalent to finding all qc structures conformal to a given structure $\eta$ (of constant {qc} scalar curvature) which also have  constant qc scalar curvature. As usual the two problems are
related by noting that on a compact quaternionic contact manifold $M$ with a fixed conformal class $[\eta]$ the {qc} Yamabe equation characterizes the non-negative extremals of the {qc} Yamabe functional.

The {4n+3 dimensional} sphere is an important example of a locally quternionic contact
conformally flat qc structure characterized locally in \cite{IV}
with the vanishing of a curvature-type tensor invariant. From the
point of view of the {qc} Yamabe problem the sphere plays a role
similar to its Riemannian and CR counterparts. A solution of the {qc}
Yamabe problem on the seven dimensional sphere equipped with its
natural quaternionic contact structure was given in \cite{IMV1}
where more details on the qc Yamabe problem can be found. The main
result of \cite{IMV1} is the following
\begin{thrm*}[\cite{IMV1}]
\label{t:Yamabe} Let $\tilde\eta=\frac{1}{2h}\eta$ be a conformal deformation of the standard qc-structure
$\tilde\eta$ on the quaternionic unit sphere $S^{7}$. If $\eta$ has constant qc scalar curvature, then up to a
multiplicative constant $\eta$ is obtained from $\tilde\eta$ by a conformal quaternionic contact automorphism.
In particular, $\lambda(S^7)= 48\, (4\pi)^{1/5}$ and this minimum value is achieved only by $\tilde\eta$ and its
images under conformal quaternionic contact automorphisms.
\end{thrm*}

Another motivation for studying the {qc} Yamabe equation and the {qc} Yamabe constant of the qc sphere comes from its connection with the
determination of the norm and extremals in a relevant
Sobolev-type embedding on the quaternionic Heisenberg group  \cite{GV} and \cite{Va} and \cite{Va2}. As well known, the {sub-Riemannian} Yamabe equation is also the
Euler-Lagrange equation of the extremals for the $L^2$ case of such embedding results. Recall the following
Theorem due to Folland and Stein \cite{FS}.
\begin{thrm*}[Folland and Stein]
\label{T:Folland and Stein} Let $\Omega \subset {G}$ be an open set in a Carnot group
${G}$ of homogeneous dimension $Q$ {\ and Haar measure $dH$}. For any $1<p<Q$ there exists
$S_p=S_p({G})>0$ such that for $u\in C^\infty_o(\Omega)$
\begin{equation}  \label{FS}
\left(\int_\Omega\ |u|^{p^*}\ dH(g)\right)^{1/p^*} \leq\ S_p\ \left(\int_\Omega |Xu|^p\ dH(g)\right)^{1/p},
\end{equation}
where $|Xu|=\sum_{j=1}^m |X_ju|^2$ with $X_1,\dots, X_m$ denoting a basis of the first layer of ${G}$
{and $p^*= \frac {pQ}{Q-p}$.}
\end{thrm*}
\noindent Let $S_p$ be the best constant in the Folland-Stein inequality, i.e., the smallest constant for which
\eqref{FS} holds.

 In \cite{IMV1} we  determined \textit{all} extremals, i.e., solutions of the qc Yamabe equation, and the best
constant in  Folland and Stein's theorem when $p=2$ in the case
of the seven dimensional quaternionic Heisenberg group.  In the case of the complex (i.e. "usual") Heisenberg group this was done earlier by Jerison and Lee \cite{JL3} who determined \emph{all} solutions to the CR Yamabe equation on the CR sphere. In that setting,  Frank and Lieb \cite{FL} determined the best constant and found all functions for which the minimum is achieved, thus simplifying parts of \cite{JL3} while answering a less general question. However, in \cite{FL}  the authors also gave sharp forms of  many Hardy-Littlewood-Sobolev type inequalities on the  Heisenberg group.

Following the idea of \cite{FL}, the main result of this paper determines   the
best constant  in the Folland and Stein's theorem when $p=2$ and the functions for which it is achieved in the case of the quaternionic Heisenberg group $\QH$  of any dimension.

As a manifold $\QH =\mathbb{H}^n\times\text {Im}\, \mathbb{H}$ with the group law
given by
\begin{equation}\label{e:heisenberg group}
( q^{\prime }, \omega^{\prime })\ =\ (q_o, \omega_o)\circ(q, \omega)\ =\ (q_o\ +\ q, \omega\ +\ \omega_o\ + \ 2\
\text {Im}\ q_o\, \bar q),
\end{equation}
\noindent where $q,\ q_o\in\mathbb{H}^n$ and $\omega, \omega_o\in
\text {Im}\, \mathbb{H}$. The standard quaternionic contact(qc)
structure is defined by the left-invariant quaternionic contact
form
$$\tilde\Theta\ =\ (\tilde\Theta_1,\ \tilde\Theta_2, \
\tilde\Theta_3)\ =\ \frac 12\ (d\omega \ - \ q^{\prime }\cdot
d\bar q^{\prime }\ + \ dq^{\prime }\, \cdot\bar q^{\prime }),$$
where $\cdot$ denotes the quaternion multiplication. The purpose
of the present note is to prove the next

\begin{thrm}\label{t:main}
\label{t:FS}
a) Let $\QH\ =\ \mathbb{H}^n\times\text {Im%
}\, \mathbb{H}$ be the quaternionic Heisenberg group. The best constant in the $L^2$ Folland-Stein embedding
inequality \eqref{FS}  is
\begin{equation*}
S_2\  =\  \frac { \left [  2^{-2n}\ \omega_{4n+3}\right ]^{-1/(4n+6)}}{2\sqrt{n(n+1)}},
\end{equation*}
where $\omega_{4n+3}=2\pi^{2n+2}/(2n+1)!$ is the volume of the unit sphere $S^{4n+3}\subset \mathbb{R}^{4n+4}$.  The non-negative functions for which \eqref{FS} becomes an equality are given by the functions of the form
\begin{equation}
\begin{aligned}
 F\ & =\ \gamma \left [(1+\lvert q \rvert^2)^2\ +\ \lvert \omega \rvert^2 \right ]^{-(n+1)}, \qquad \gamma=const,
\end{aligned}
\end{equation}
and all functions obtained from $F$ by translations \eqref{translation} and dilations \eqref{scaling}.

b) The {qc} Yamabe constant of the standard qc structure  of the sphere is
\begin{equation}\label{e:yamabe constant}
\lambda(S^{4n+3}, [\tilde\eta])\  =\ 16\, n(n+2)\, \left [ \left ( (2n)! \right) \omega_{4n+3}\right ]^{1/(2n+3)}.
\end{equation}
\end{thrm}

These constants are in complete agreement with the ones
obtained in \cite{IMV1} and \cite{GV} taking into account the next Remark and the well known formulas  involving the  gamma function
\begin{equation*}
\begin{aligned}
& \Gamma(n+1)=n!, \qquad \Gamma (z+n) = z(z+1)\dots(z+n-1)\Gamma (z),\qquad n\in \mathbb{N},\\
& \omega_n= 2\pi^{n/2}/ \Gamma(n/2) \qquad\text{-- volume of unit $(n-1)$--dimensional sphere in } \mathbb{R}^n,\\
& \Gamma(2z)=2^{2z-1}\ \pi^{-1/2}\ \Gamma(z)\, \Gamma \left (z+\frac 12\right ) \quad\text{  -- the  Legendre formula. }
\end{aligned}
\end{equation*}
{\ Our result partially confirms the Conjecture made after \cite[Theorem 1.1]{GV}.  In addition, the fact that any function of the described form is a solution of Yamabe problem was first noted in \cite{GV2} in the setting of groups of Heisenberg type. Of course, this class of groups is much wider than the class of groups of Iwasawa type.
\begin{rmrk}
\label{rem111} {\ With the left invariant basis of Theorem~\ref{t:FS} the quaternionic Heisenberg group
 is not a group of Heisenberg type. If we consider it
as a group of Heisenberg type then the best constant in the $L^2$ Folland-Stein embedding theorem is, cf.
\cite[Theorem 1.6]{GV},
\begin{equation*}
S_2\ =\ \frac{1}{\sqrt{4n(4n + 4)}}\ 4^{3/(4n+6)}\ \pi^{-(4n+3)/2(4n+6)}\ \left(\frac{\Gamma(4n + 3)}{\Gamma((4n + 3)/2)}\right)^{1/(4n+6)} .
\end{equation*}
\noindent and extremals are given by dilations and translations of the function
\begin{equation*}
F(q,\omega)\ =\ \gamma\ \left[(1 + |q|^2)^2\ +\ 16 |\omega|^2)\right]^{-(n+1)}, \ (q,\omega)\in \QH.
\end{equation*}
}
\end{rmrk}
It is worth pointing that studying the Yamabe extremals in the sub-Riemannian setting has applications to sharp inequalities in the Euclidean setting. For example, in the paper \cite{Va3} are determined the extremals of some  \textit{Euclidean} Hardy-Sobolev inequalities  involving the distance to a $n-k$ dimensional coordinate subspace of $\Rn$. This is  achieved  by relating extremals on the Heisenberg groups to extremals in the Euclidean setting.  In the particular case when $k=n$  one obtains the Caffarelli-Kohn-Nirenberg
inequality, see \cite{CKN}, for which  the optimal constant was found in
\cite{GY}.

\begin{conv}
\label{conven} We use the following conventions:
\begin{itemize}
\item the abbreviation \textit{qc} will stand for quaternionic contact;
\item  $\QH$ will denote the qc Heisenberg group;
\item $\tilde\eta$ will denote the standard qc form on the unit sphere $S^{4n+3}$, see
\eqref{e:stand cont form on S}. Note that this form is actually twice the 3-Sasakain qc form on $S^{4n+3}$;
\item $Vol_{\eta}$ will denote the volume form determined by the qc form $\eta$, thus $Vol_{\eta}=\eta_1\wedge\eta_2\wedge\eta_3\wedge (\omega_1)^{2n}$, see \cite[Chapter 8]{IMV}.
\end{itemize}
\end{conv}

{\bf Acknowledgments.} Research was partially supported by
Contract ``Idei", DO 02-257/18.12.2008 and DID 02-39/21.12.2009.
S.I. and I. M. are partially supported by the Contract 082/2009
with the University of Sofia `St.Kl.Ohridski'.

\section{The model quaternionic contact structures}
In this section we review the standard quaternionic contact structure on the quaternionic  Heisenberg group and the
$4n+3$-dimensional unit sphere. We will rely heavily on \cite{IMV}, but prefer to repeat some key points in
order to make the current paper somewhat self-contained. Besides serving as a background, this section will supply some key numerical constants - the {qc} scalar curvature and the first eigenvalue of the {sub-laplacian of the standard qc form of the sphere}.  This will be achieved using the conformal {sub-}laplacian and the properties of the Cayley transform.

First let us recall  the quaternionic Heisenberg group \cite[Section 5.2]{IMV}. We remind the following model of
the quaternionic Heisenberg group  $\QH$. Define $\QH \ =\ \Hn\times\text {Im}\, \mathbb{H}$ with the group law
given by $
 ( q', \omega')\ =\ (q_o, \omega_o)\circ(q, \omega)\ =\ (q_o\ +\ q, \omega\ +\ \omega_o\ + \ 2\ \text
{Im}\  q_o\, \bar q),
$ where $q,\ q_o\in\Hn$ and $\omega, \omega_o\in \text {Im}\, \mathbb{H}$. In
coordinates, with $\omega=ix+jy+kz$ and $q_\alpha=t_\alpha+ ix_\alpha +jy_\alpha +kz_\alpha$, $\alpha=1,\dots n$, a basis of left invariant horizontal vector fields $T_{\alpha},
X_{\alpha}=I_1T_{\alpha},Y_{\alpha}=I_2T_{\alpha},Z_{\alpha}=I_3T_{\alpha}, \alpha =1\dots,n$ is given by
\begin{equation*}
\begin{aligned}
T_{\alpha}\  =\  \dta {}\ +2x_{\alpha}\dx {}+2y_{\alpha}\dy {}+2z_{\alpha}\dz {} \,\qquad
  X_{\alpha}\  =\  \dxa {}\ -2t_{\alpha}\dx {}-2z_{\alpha}\dy
{}+2y_{\alpha}\dz {} \,\\
Y_{\alpha}\  =\  \dya {}\ +2z_{\alpha}\dx {}-2t_{\alpha}\dy {}-2x_{\alpha}\dz {}\,\qquad Z_{\alpha}\  =\  \dza
{}\ -2y_{\alpha}\dx {}+2x_{\alpha}\dy {}-2t_{\alpha}\dz {}\,.
\end{aligned}
\end{equation*}
The above vectors generate the \emph{horizontal space}, denoted as usual by $H$.  In addition,  by declaring them to be an orthonormal basis we obtain a metric on the horizontal space, which is the so called \emph{horizontal metric}.
\noindent The central (vertical) vector fields $\xi_1,\xi_2,\xi_3$ are described as
follows\\
\centerline{$\xi_1=2\dx {}\,\quad \xi_2=2\dy {}\,\quad \xi_3=2\dz {}\,.$}
The standard quaterninic contact form, written as a purely imaginary quaternion valued form $\tilde\Theta=i\tilde\Theta_1
+ j\tilde\Theta_2 +k\tilde\Theta_3)$, is
\begin{equation*}
2\tilde\Theta\  =\ \ d\omega \ - \ q' \cdot d\bar q' \ + \ dq'\, \cdot\bar q',
\end{equation*}
where $\cdot$ denotes the quaternion multiplication. {The Biquard connection coincides with the flat left invariant connection on $\QH$, in particular the qc scalar curvature vanishes.}

Following \cite{IMV}, we give another model of the Heiseneberggroup, which is the one we will use in this paper.
Let us identify $\QH$ with the boundary $\Sigma$ of a Siegel domain in $\Hn\times\mathbb{H}$,
\[
\Sigma\ =\ \{ (q',p')\in \Hn\times\mathbb{H}\ :\ \Re {\ p'}\ =\ \abs{q'}^2 \},
\]
by using the map $(q', \omega')\mapsto (q',\abs{q'}^2 - \omega')$. Since \hspace{3mm}
$dp'\ =\ q'\cdot d\bar q'\ +\ dq'\, \cdot\bar {q}'\ -\ d\omega',
$ 
\hspace{3mm} under the identification  of $\QH$ with $\Sigma$ we have also \hspace{3mm}
$2\tilde{\Theta}\ =\ - dp'\ +\ 2dq'\cdot\bar {q}'.
$ 
\hspace{3mm} Taking into account that $\tilde{\Theta}$ is purely imaginary, the last equation can be written
also in the following form
\[
4\,\tilde{\Theta}\ =\ (d\bar p'\ -\ d p')\ +\ 2dq'\cdot\bar {q'}\ -\ 2q'\cdot d\bar q'.
\]
Now, consider the Cayley transform, see \cite{Ko1} and \cite{CDKR}, as the map  $\ \mathcal{C}:S\mapsto \Sigma\ $ from the
sphere $S\ =\ \{\abs{q}^2+\abs{p}^2=1 \}\subset \Hn\times\mathbb{H}$ minus a point to the Heisenberg group $
\Sigma$, with $\mathcal{C}$ defined by
\[
 (q', p')\ =\ \mathcal{C}\ \Big ((q, p)\Big), \qquad
q'\ =\ (1+p)^{-1} \ q, \qquad p'\ =\ (1+p)^{-1} \ (1-p)
\]
\noindent and with an inverse map $(q, p)\ =\ \mathcal{C}^{-1}\Big ((q', p')\Big)$ given by
\[
q\ =\ \ 2(1+p')^{-1} \ q', \qquad  p\ =\ (1+p')^{-1} \ (1-p').
\]
\noindent The Cayley transform maps $S^{4n+3}\setminus \{(-1,0) \}$, $(-1,0)\in \Hn\times\mathbb{H}$, to
$\Sigma$ since
$$
\Re {\ p'}\ =\ \Re { \frac {(1+\bar p) (1-p)} {\abs {1+p\,}^2}
 }
 \ =\ \Re { \frac {1- \abs  {p} } {\abs
{1+ p\,}^2} }\ =\ \frac {\abs{q}^2}{\abs {1+p\,}^2}\ =\ \abs {q'}^2.
$$
Writing the Cayley transform in the form \hspace{3mm} $ (1+p)q'\ =\  \ q, \quad (1+p)p'\ =\  1-p, $ \hspace{3mm}
gives
\[
dp\cdot q'\ +\ (1+p)\cdot dq'\ =\ d q, \hskip.5truein dp\cdot p'\ +\ (1+p)\cdot dp'\ =\ -dp,
\]
from where we find
\begin{equation*}
\begin{aligned}
dp'\ & =\ -2(1+p)^{-1}\cdot dp \cdot (1+p)^{-1}\\   dq' \ & =\ (1+p)^{-1}\cdot [ dq\ -\ dp\cdot (1+p)^{-1}\cdot
q ].
\end{aligned}
\end{equation*}
\noindent The Cayley transform is a conformal\ quaternionic contact diffeomorphism between the quaternionic
Heisenberg group with its standard quaternionic contact structure $\tilde\Theta$ and the sphere minus a point
with its standard structure $\tilde\eta$.  In fact, by \cite[Section 8.3]{IMV} we have
\begin{equation*}
\Theta\ \overset{def}{=}\ \lambda\ \cdot (\mathcal{C}^{-1})^*\, \tilde\eta\ \cdot \bar\lambda\ =\ \frac
{8}{\abs{1+p'\, }^2}\, \tilde\Theta.
\end{equation*}
\noindent where $\lambda\ = {\abs {1+p\,}}\, {(1+p)^{-1}}$ is a unit quaternion and $\tilde\eta$ is the standard
contact form on the sphere,
\begin{equation}\label{e:stand cont form on S}
\tilde\eta\ =\ dq\cdot \bar q\ +\ dp\cdot \bar p\ -\ q\cdot d\bar q -\ p\cdot d\bar p.
\end{equation}

\begin{lemma}\label{l:scal of sphere}
The qc scalar curvature $\tilde S$ of the standard qc structure \eqref{e:stand cont form on S} on $S^{4n+3}$ is
\begin{equation}\label{e:Scal of st sphere}
\tilde S\ =\ \frac 12 (Q+2)(Q-6)=8n(n+2).
\end{equation}
\end{lemma}

\begin{rmrk}\label{r:standctct form}
Notice that the standard qc contact form we consider here is twice the 3-Sasakian form on $S^{4n+3}$, {which has qc scalar curvature equal to 16n(n+2) \cite{IMV}}.
\end{rmrk}

\begin{proof}
Let us introduce the functions
\begin{equation}
\begin{aligned}\label{e:h and Phi}
 h& =\frac  {1}{16}|1+p'|^2=\frac {1}{16}\left [(1+|q'|^2)^2+
 |\omega'|^2 \right ], \qquad (q',p')\in \Sigma\subset \Hn\times\mathbb{H}, \quad p'=|q'|^2+ \omega',\\
& \text{and}\\
\Phi & =\left({2h} \right)^{-(Q-2)/4} = 8^{(Q-2)/4}\ {\left [(1+|q'|^2)^2+ |\omega'|^2 \right ]}^{-(Q-2)/4},
\end{aligned}
\end{equation}
so that now we have $$\Theta=\frac {1}{2h}\tilde\Theta =\Phi^{4/(Q-2)}\tilde\Theta.$$ With the help of
\cite[Section 5.2]{IMV}  a small calculation shows that the sub-laplacian of $h$ w.r.t. $\tilde\Theta$ is given
by $\triangle h = \frac {Q-6}{4} + \frac {Q+2}{4}|q'|^2$ and thus  $\Phi$ is a solution of the qc Yamabe equation
on the  Heisenberg group $\Sigma$
\begin{equation}\label{e:Yamabe for Phi}
\triangle \Phi = -K\, \Phi^{2^*-1}, \qquad K=(Q-2)(Q-6)/8,
\end{equation}
where $\triangle$ is the sub-laplacian on the quaternionic Heisenberg group. Denoting with $\lap$ and $\tilde \lap$ the
conformal sub-laplacians of $\Theta$ and $\tilde\Theta$, respectively, we have
$$\Phi^{-1}\lap (\Phi^{-1}u) = \Phi^{-2^*}\tilde \lap u.$$ We remind, cf. \cite{Biq1} and \cite{IMV1}, that  for a qc contact form $\Theta$ the conformal sublaplacian is,
\begin{equation*}
\lap = a\triangle_{\Theta} - S_{\Theta}, \qquad a=4\frac {Q+2}{Q-2},
\end{equation*}
where $\triangle_{\Theta}$ is the sub-laplacian  associated to $\Theta$, i.e., $\triangle_{\Theta} u=tr (\nabla^{\Theta}d u)$--the horizontal trace of the Hessian of $u$, using the Biquard connection $\nabla^{\Theta}$ of $\Theta$, and $ S_{\Theta}$ is the qc scalar curvature of $\Theta$.
Thus, letting $u=\Phi$ we come to $\lap ( 1 ) = \Phi^{1-2^*}\tilde \lap \Phi$, which shows $- S_{\Theta}=-4\frac {Q+2}{Q-2}K$. The latter is the same as  that of $\tilde \eta$ since the two structures are isomorphic via the diffemorphism $\mathcal{C}$, or rather its extension, since we can consider $\mathcal{C}$ as a quaternionic contact conformal transformation between the whole sphere $ S^{4n+3}$  and the compactification $\hat{\Sigma}\cup {\infty}$ of the quaternionic Heisenberg group by adding the point at infinity, cf. \cite[Section 5.2]{IMV1}.
\end{proof}

We turn to the task of determining the first eigenvalue of the
sub-laplacian on $S^{4n+3}$. In fact, we shall need only the fact
that the restriction of every coordinate function  is an
eigenvalue.  The proof of this fact can be seen directly without
any reference to the Biquard connection, but this will require
setting a lot of notation, so we prefer to use a result from
\cite{IMV}.
\begin{lemma}
If $\zeta$ is any of the (real) coordinate functions in $\mathbb{R}^{4n+4}=\Hn\times\mathbb{H}$, then
\begin{equation}\label{e:eigen functions}
\tilde \triangle \zeta = -\lambda_1 \zeta, \quad \lambda_1=\frac {\tilde S}{Q+2}=  {2n}
\end{equation}
for the horizontal trace of the Hessian, where $\tilde \triangle$ is the sub-laplacian of the standard qc form $\tilde \eta$  of $S^{4n+3}$.
\end{lemma}

\begin{proof}
It is enough to furnish a proof for the sub-laplacian on the 3-Sasakain sphere since the two qc forms defer by a
constant. We can see that every $\zeta$ of the considered type is an eigenfunction by using \cite[Corollary
6.24]{IMV}. It will be enough to see it for one coordinate function provided the sub-laplacian on the sphere is
rotation invariant. Thus, let us  take $\zeta=t_1$. Notice that $\zeta$ is quaternionic pluri-harmonic
\cite[Definition 6.7]{IMV} since it is the real part of the anti-regular  function $t_1+ix_1 - jy_1-kz_1$. So,
its restriction to the 3-Sasakain sphere is the real part of an anti-CRF function. Therefore we  apply
\cite[Corollary 6.24]{IMV} which gives $tr(\nabla d\zeta) = 4\lambda n$ for the sub-laplacian of the 3-sasakain qc structure on the sphere. Next, we compute $\lambda$, which can
be found in \cite[Theorem 6.20]{IMV}. Using that the sphere is 3-Sasakian it follows the Reeb vector fields are
obtained from the outward pointing unit normal vector $N$ as follows, $\xi_1=i N$, $\xi_2=jN$ and $\xi_3=k N$,
where for a point on the sphere we have $N(q)=q\in \mathbb{H}^{n+1}$. Therefore $\lambda=-t_1=-\zeta$.  To see
this easier notice that only the first four coordinates of $N$ matter. So, if we assume $n=0$ we have
$iN=-x+it+ky-jz$, $jN=-y+iz+jt-kx$ and $kN=-z-iy+jx+kt$, so we need to sum the \emph{real dot product }of these
vectors with $i$, $j$ and $k$, respectively, which gives $-t$. Thus, for the sub-laplacian on the 3-Sasakian sphere we have
\begin{equation*}
tr(\nabla d\zeta) = -4n\zeta,
\end{equation*}
where $\zeta$ is the restriction any of the coordinate functions of $\mathbb{R}^{4n+4}=\Hn\times\mathbb{H}$.
Since the qc contact form $\tilde\Theta$ is {twice} the 3-Sasakain qc contact form on the sphere it follows $\tilde
\triangle$ is $1/2$ of the 3-Sasakain sub-laplacian. Thus
\begin{equation*}
\tilde\triangle = -2n\zeta,
\end{equation*}
which shows $\lambda_1=2n=\frac 12(Q-6)= {\tilde S}/{(Q+2)}$.
\end{proof}

We finish this section with a simple Lemma which will be used to
relate the various explicit constants.  Its claim also  follows
from the conformal invariance of the Yamabe equation, but we
prefer to give a proof, which is independent of the notion of qc
scalar curvature. We recall, see \cite[Chapter 8]{IMV}, that $Vol_{\eta}$ will denote the
volume form determined by the qc form $\eta$, thus
$Vol_{\eta}=\eta_1\wedge\eta_2\wedge\eta_3\wedge(\omega_1)^{2n}$. Also, for a qc form $\eta$ we let $|\nabla^{\eta} F|^2= \sum_{\alpha=1}^{4n} |dF(e_\alpha)|^2$  be the square of the length of the horizontal gradient of {a function} $F$ taken with respect to an orthonormal basis of the horizontal space $H= Ker \, \eta$ and {the} metric determined by  $\eta$.
\begin{lemma}\label{l:invariance}
Let $F\in\domG$, cf. \eqref{e:sobolev space}, be a positive
function with $\int_{\bG}\ F^{2^*} Vol_{\tilde\Theta}=1$. Then we have
\begin{equation}\label{e:invariance of dirichlet 2}
\int_{\bG}\ a|\nabla^{\tilde\Theta} F|^2\ Vol_{\tilde\Theta}\ =\ \int_{S^{4n+3}}\  \left ( a|\nabla^{\tilde\eta} g|^2 + \tilde S g^2\right )\ Vol_{\tilde\eta},\qquad a=4(2^*-1),
\end{equation}
and $$\int_{\bG}\ g^{2^*} Vol_{\tilde\eta}=1,$$
where
\begin{equation}\label{e:g def}
g=\mathcal{C}^*(F\Phi^{-1}),
\end{equation}
and, as
before, $\mathcal{C}:S^{4n+3}\rightarrow \Sigma$ is the Cayley transform, $\Theta =\Phi^{4/(Q-2)}\tilde\Theta$, cf. \eqref{e:h and Phi}.
\end{lemma}

\begin{rmrk}\label{r:volume form vs lebesgue}
Notice that $Vol_{\tilde\Theta}=2^{-3}\,(2n)!\, dH$, where $dH$ is the Lebesgue measure in $\mathbb{R}^{4n+3}$, which is a Haar measure on the group.
\end{rmrk}

\begin{proof}
It will be convenient for the remaining of
this proof to denote by small letters the pull-back by the Cayley transform of a function denoted with the
corresponding capital letter.  Thus, $f=\mathcal{C}^*F=F\circ \mathcal{C}$, $\phi=\mathcal{C}^*(\Phi)$ and $g=f\phi^{-1}$.
By the conformality of the qc structures on the group and the sphere we have
\begin{equation}\label{e:Cayley volumes}
Vol_{\Theta}=\Phi^{2^*}  Vol_{\tilde\Theta}
\end{equation}
By \eqref{e:Cayley volumes} we have $F^{2^*} Vol_{\tilde\Theta} = f^{2^*}
\phi^{-2^*}  Vol_{\tilde\eta}$, which motivates the definition \eqref{e:g def} of the function $g$
which is defined on the sphere and should be regarded as corresponding to the function $F$. Thus, we have for
example $F=G\Phi$. By definition we have $$\int_{\bG}\ g^{2^*} Vol_{\tilde\eta}=1,$$ so our next task is to see
that the Yamabe integral is preserved
\begin{equation}\label{e:invariance of dirichlet}
\int_{\bG}\ |\nabla^{\tilde\Theta} F|^2\ Vol_{\tilde\Theta}\ =\  \int_{S^{4n+3}}\ \left ( |\nabla^{\tilde\eta} g|^2 +K g^2  \right ) \ Vol_{\tilde\eta}.
\end{equation}
Here is where we shall exploit that a power of the  conformal factor of the Cayley transform is a solution of
the Yamabe equation.  Let $\left <\nabla^{\Theta} \Phi,\nabla^{\Theta} G \right > = \sum_{a=1}^{4n}(e_a \Phi)\, (\e_a G)$ where
$\{e_1,\dots,e_{4n}\}$ is an orthonormal basis of the horizontal space $H$. Using the divergence formula from
\cite[Section 8.1]{IMV} we find
\begin{multline*}
\int_{\bG}\ |\tilde\nabla^{\Theta} F|^2\ Vol_{\tilde\Theta} = \int_{\bG}\ |\nabla^{\tilde\Theta} (G\Phi)|^2\ Vol_{\tilde\Theta}
= \int_{\bG}\ \Big(G^2|\nabla^{\tilde\Theta} \Phi|^2 + \Phi^2|\nabla^{\tilde\Theta} G|^2 + \left <\Phi\nabla^{\tilde\Theta} \Phi,\nabla^{\tilde\Theta} G^2 \right >\Big) \ Vol_{\tilde\Theta}\\
= \int_{\bG}\ \Big( \Phi^2|\nabla^{\tilde\Theta} G|^2  - G^2\Phi \triangle_{\tilde\Theta} \Phi\ \Big) Vol_{\tilde\Theta}.
\end{multline*}
Now, the Yamabe equation \eqref{e:Yamabe for Phi} gives
\begin{multline*}
\int_{\bG}\ |\nabla^{\tilde\Theta} F|^2\ Vol_{\tilde\Theta} = \int_{\bG}\ \left( \Phi^2|\nabla^{\tilde\Theta} G|^2  +K G^2\Phi^{2^*}\ \right)Vol_{\tilde\Theta}\\
=\int_{S^{4n+3}}\ \left ( \phi^{2-2^*}(|\nabla^{\tilde\Theta} G|\circ \mathcal{C})^2  +K g^2 \ \right) Vol_{\tilde\eta}= \int_{S^{4n+3}}\ \left(|\nabla^{\tilde\eta} g|^2  +K g^2\ \right) Vol_{\tilde\eta},
\end{multline*}
taking into account that $\mathcal{C}$ is
a qc conformal map. Finally, a glance at \eqref{e:Yamabe for Phi} and \eqref{e:Scal of st sphere} shows  $\tilde S/K=4(2^*-1)=(4(Q+2)/(Q-2)$ which allows to put  \eqref{e:invariance of dirichlet} in the form \eqref{e:invariance of dirichlet 2}.
\end{proof}

\section{The best constant in the Folland-Stein inequality}
In this section, following \cite{FL}, we prove the main Theorem.
 It is important to observe that a suitable adaptation of the method of concentration of compactness due
 to P. L. Lions \cite{L3}, \cite{L4} allows to prove that in any Carnot group the Yamabe constant and optimal constant in the Folland-Stein inequality is achieved in the space $\domG$, see \cite{Va} and \cite{Va2}. Here $$
\domG =\overline{C^\infty_o(\QH)}^{||\cdot||_{\domG}}.$$ The  space $\domG$ is endowed with the  norm
\begin{equation}\label{e:sobolev space}
||u||_{\domG} \ =\ ||\, |\nabla u|\, ||_{L^{2^*}(\QH)} .
\end{equation}
where $\nabla u$ is the horizontal gradient of $u$ and $|\nabla u|^2=\sum_{a=1}^{4n} (e_au)^2$ for an orthonormal basis $\{e_1,\dots, e_{4n}\}$ of horizontal left invariant vector fields.

  In this regard an elementary, yet crucial
 observation, is that if $u$ is an entire solution to the Yamabe equation, then such are also the two functions
\begin{equation}\label{translation}
\tau_h u\ \overset{def}{=}\  u\circ \tau_h, \quad\quad \quad h\in \bG ,
\end{equation}
where $\tau_h:\bG\to \bG$ is the operator of left-translation $\tau_h(g) = hg$, and
\begin{equation}\label{scaling}
u_\lambda\ \overset{def}{=}\ \lambda^{(Q-2)/2}\ u\circ \delta_\lambda, \quad\quad\quad \lambda >0.
\end{equation}
The  Heisenberg dilations are defined  by $$\delta_\lambda\left ( (q', \omega')\right)=\left ( (\lambda q', \lambda^2 \omega')\right), \qquad (q', \omega')\in \QH$$
It is also well known, \cite{Va} and \cite{Va2}, that there are smooth positive minimizer of the Folland-Stein inequality on the quaternionic Heisenberg group $\QH$.  These facts will be used without further notice on regularity and existence.

We start with the "new" key, see  \cite{BFM} and \cite{FL}, allowing the ultimate solution of the considered problem.

\begin{lemma}\label{l:hersch}
For every $v\in L^1(S^{4n+3})$ with $\int_{S^{4n+3}} v \ Vol_{\tilde\eta}=1$ there is a quaternionic contact
conformal transformation $\psi$ such that $$\int_{S^{4n+3}} \psi\, v \ Vol_{\tilde\eta} =0.$$
\end{lemma}

\begin{proof}
Let $P\in S^{4n+3}$ be any  point of the quaternionic sphere and $N$ be its antipodal point. Let us consider
the local coordinate system near $P$ defined by the Cayley transform $\mathcal{C}_N$ from $N$. It is known that
$\mathcal{C}_N$ is a quaternionic contact conformal transformation between $ S^{4n+3}\setminus {N}$  and the
quaternionic Heisenberg group. Notice that in this coordinate system $P$ is mapped to the identity of the group.
For every $r$, $0<r<1$, let $\psi_{r,P}$ be the qc conformal transformation of the sphere, which in the fixed
coordinate chart is given on the group by a dilation  with center the identity by a factor $\delta_{r}$. If we
select a coordinate system in $\mathbb{R}^{4n+4}=\Hn\times\mathbb{H}$ so that $P=(1,0)$ and $N=(-1,0)$ and then
apply the formulas for the Cayley transform from \cite[Section 8.2]{IMV} the formula for
$(q^*,p^*)=\psi_{r,P}(q,p)$ becomes
\begin{equation*}
\begin{aligned}
q^* & =2r\left (  1 + r^2(1+p)^{-1} (1-p) \right )^{-1} \left ( 1+p \right ) q\\
p^* & =\left ( 1 +r^2(1+p)^{-1}(1-p) \right )^{-1} \left (  1-r^2(1+p)^{-1}(1-p) \right ), i.e,
\end{aligned}
\end{equation*}

We can define then the map $\Psi: B\rightarrow \bar B$, where $B$  ( $\bar B$ ) is the open (closed) unit ball
in $\mathbb{R}^{4n+4}$, by the formula $$\Psi(rP)=\int_{S^{4n+3}} \psi_{1-r,P}\, v\  Vol_{\tilde\eta}.$$ Notice
that $\Psi$ can be continuously extended to $\bar B$ since for any point $P$ on the sphere, where $r=1$, we have
$\psi_{1-r,P}(Q)\rightarrow P$ when $r\rightarrow 1$. In particular, $\Psi=id$ on $ S^{4n+3}$. Since the sphere
is not a homotopy retract of the closed ball it follows that there are $r$ and $P\in S^{4n+3}$ such that
$\Psi(rP)=0$, i.e., $\int_{S^{4n+3}} \psi_{1-r,P}\,v\  Vol_{\tilde\eta}=0$. Thus, $ \psi=\psi_{1-r,P}$ has the
required property.
\end{proof}

In the  next step we prove that we can assume that the minimizer of the Folland-Stein inequality satisfies the
zero center of mass condition.  A number of well known invariance properties of the Yamabe functional will be
exploited.

\begin{lemma}\label{l:zero mass is enough}
Let $v$ be a smooth positive function on the sphere with $\int_{S^{4n+3}} v^{2^*}\, Vol_{\tilde\eta}=1$. There is a smooth positive function $u$ such that $\int_{S^{4n+3}}\Bigl(4\frac {Q+2}{Q-2}\ \lvert \nabla u \rvert^2\ +\ \tilde{S}\, u^2\Bigr)\, Vol_{\tilde\eta}\  = \int_{S^{4n+3}}\Bigl(4\frac {Q+2}{Q-2}\ \lvert \nabla v \rvert^2\ +\ \tilde{S}\, v^2\Bigr)\, Vol_{\tilde\eta}$ and $\int_{S^{4n+3}} u^{2^*}\, Vol_{\tilde\eta}=1$.
In addition,
\begin{equation}\label{e:zero mass}
\int_{S^{4n+3}} P\, u^{2^*}(P) \, Vol_{\tilde\eta}=0, \qquad P\in \mathbb{R}^{4n+4}=\Hn\times\mathbb{H}.
\end{equation}
In particular, the Yamabe constant
\begin{equation}\label{e:var problem}
\lambda(S^{4n+3}, [\tilde\eta])\ =\ \inf \{ \int_{S^{4n+3}}\Bigl(4\frac {Q+2}{Q-2}\ \lvert \nabla v \rvert^2\ +\ \tilde{S}\, v^2\Bigr)
Vol_{\tilde\eta} :\ \int_{S^{4n+3}}  v^{2^*}\ Vol_{\tilde\eta} \ =\ 1, \ v>0  \}
\end{equation}
is achieved for a positive function $u$ with a zero center of mass, i.e., for a function $u$ satisfying \eqref{e:zero mass}.
\end{lemma}

\begin{proof}
By  \cite[Section 8.1]{IMV}, $Vol_{\eta}=\eta_1\wedge\eta_2\wedge\eta_3\wedge (\omega_1)^{2n}$ is a volume form
on a qc manifold with contact form $\eta$. Thus if $\eta$ is a qc structure on the sphere which is qc conformal
to the standard qc structure $\tilde\eta$,  $\eta=\phi^{4/(Q-2)} \tilde \eta$, then $Vol_{\eta}=\phi^{2^*}
Vol_{\tilde\eta}$. This allows to cast equation \eqref{e:conf Yamabe} in the form $$\phi^{-1}v\,
\mathcal{L}(\phi^{-1}v)\ Vol_{\eta}=v\mathcal{\tilde L}(v)\ Vol_{\tilde\eta}.$$ Therefore, if we take a positive function
$v$ on the sphere $\int_{S^{4n+3}}  v^{2^*}\, Vol_{\tilde\eta}=1$ and then consider the function
\begin{equation}\label{e:zero mass transform}
u=\phi^{-1}(v\circ
\psi^{-1}),
 \end{equation}
 where $\psi$ is the qc conformal map of Lemma \ref{l:hersch}, $\eta\equiv(\psi^{-1})^*\tilde\eta$, and $\phi$ is the corresponding conformal factor of $\psi$, we can see that $u$ achieves the claim of the Lemma.
\end{proof}

We shall call a function $u$ on the sphere a \textit{well centered} function when \eqref{e:zero mass} holds true.
In the next step we show that a well centered minimizer has to be constant.
\begin{lemma}\label{l:const mnimizer}
If $u$ is a well centered local minimum of the problem \eqref{e:var problem}, then $u\equiv const$.
\end{lemma}

\begin{proof}
Let $\zeta$ be a smooth function on the sphere $S^{4n+3}$. After applying the divergence formula \cite[Section
8]{IMV} we obtain the formula
\begin{equation}\label{e:Upsilon for zeta u}
\Upsilon (\zeta u) = \int_{S^{4n+3}}\zeta^2 \Bigl(4\frac {Q+2}{Q-2}\ \lvert \tilde\nabla u \rvert^2\ +\
\tilde{S}\, u^2\Bigr)\  Vol_{\tilde\eta} \ - \ 4\frac {Q+2}{Q-2}\int_{S^{4n+3}} u^2 \zeta{\tilde \triangle}
\zeta \ Vol_{\tilde\eta}.
\end{equation}
This suggests to take as a test function $\zeta$ an eigenfunction of the sub-laplacian ${\tilde \triangle}$ of
the standard qc structure.  In particular, we can let $\zeta$ be any of the coordinate functions in
$\Hn\times\mathbb{H}$ in which case ${\tilde \triangle} \zeta =-\lambda_1 \zeta$.

It will be useful to introduce the functional $N(v)=\left ( \int_{S^{4n+3}}  v^{2^*}\ Vol_{\tilde\eta}\right
)^{2/2^*}$ so that
\begin{equation}\label{e:Yamabe energy}
\lambda(S^{4n+3},[\tilde\eta]) = \inf \{ \mathcal{E}(v):\, v\in D \, (S^{4n+3}) \},\qquad \mathcal{E}(v)\overset{def}{=}\Upsilon ( v) / N(v).
\end{equation}
Computing the second variation $\delta^2 \mathcal{E}(u)v = \frac {d^2}{dt^2} \mathcal{E}(u+tv)_{|_{t=0}}$ of $\mathcal{E}(u)$ we see that the local minimum condition $\delta^2 \mathcal{E}(u)v\geq 0$ implies
\begin{equation*}
\Upsilon ( v) -(2^*-1)\Upsilon ( u)\int_{S^{4n+3}}  u^{2^*-2}v^2\ Vol_{\tilde\eta}  \ \geq 0
\end{equation*}
for any function $v$ such that $\int_{S^{4n+3}}  u^{2^*-1}v\ Vol_{\tilde\eta}=0$.
Therefore, for $\zeta$ being any of the coordinate functions in $\Hn\times\mathbb{H}$ we have
$$
\Upsilon (\zeta u) -(2^*-1)\Upsilon ( u)\int_{S^{4n+3}}  u^{2^*}\zeta^2\ Vol_{\tilde\eta}  \ \geq 0,
$$
which after summation over all coordinate functions taking also into account \eqref{e:Upsilon for zeta u} gives
\begin{equation*}
\Upsilon (u)  - (2^*-1)\Upsilon ( u) -4\lambda_1(2^*-1)   \int_{S^{4n+3}} u^2 \ Vol_{\tilde\eta}\geq 0,
\end{equation*}
which implies, recall $2^*-1=( {Q+2})/({Q-2})$,
\begin{multline*}
0\leq 4(2^*-1)\left ( 2^*-2\right)\int_{S^{4n+3}} |\tilde \nabla u|^2\ Vol_{\tilde\eta} \\
\leq \ \left (4\lambda_1(2^*-1) -\left (2^*-2\right )\tilde S \right )  \int_{S^{4n+3}}  u^{2^*}\
Vol_{\tilde\eta}.
\end{multline*}
Thus, our task of showing that $u$ is constant will be achieved once we see that
\begin{equation}\label{e:key ineq}
4\lambda_1(2^*-1) -\left (2^*-2\right )\tilde S \leq 0,\  \text{ i.e, }\  \lambda_1\leq \tilde S/(Q+2).
\end{equation}
 By Lemma \ref{e:eigen functions} we
 have actually equality $\lambda_1= {\tilde S}/{(Q+2)}$,
which completes the proof.  It is worth observing that inequality \eqref{e:key ineq} can be written in the form
\begin{equation*}
\lambda_1\, a \leq (2^*-2) \, \tilde S,
\end{equation*}
where $a$ is the constant in front of the (sub-)laplacian  in the conformal (sub-)laplacian, i.e., $a=4\frac {Q+2}{Q-2}$ in our case.
\end{proof}

At this point the proof of our main Theorem \ref{t:main} follows easily as follows.
\begin{proof}[Proof of Theorem \ref{t:main}]
 Let $F$ be a minimizer (local minimum) of the Yamabe functional $\mathcal{E}$ on $\QH$ and $g$ the corresponding function on the sphere defined in Lemma \ref{l:invariance}. By Lemma \ref{l:zero mass is enough} and \eqref{e:zero mass transform} the function $g_0=\phi^{-1}(g\circ
\psi^{-1})$ will be well centered and a  minimizer (local minimum) of the Yamabe functional $\mathcal{E}$ on $S^{4n+3}$. The latter claim uses also the fact that the map $v\mapsto u$ of equation \eqref{e:zero mass transform} is one-to-one and onto on the space of smooth positive functions on the sphere. Now, from Lemma  \ref{l:const mnimizer}  we conclude that $g_o=const$. Looking back at the corresponding functions on the group we see that
\begin{equation*}
F_0 =\gamma\,{\left [(1+|q'|^2)^2+ |\omega'|^2 \right ]}^{-(Q-2)/4}
\end{equation*}
for some $\gamma=const.>0.$
Furthermore, the proof of Lemma \ref{l:hersch} shows that $F_0$ is obtained from $F$ by a translation \eqref{translation} and dilation \eqref{scaling}. Correspondingly, any positive minimizer (local maximum) of problem \eqref{S}  is given up to dilation or translation by the function
\begin{equation}\label{e:extremals}
F =\gamma\,{\left [(1+|q'|^2)^2+ |\omega'|^2 \right ]}^{-(Q-2)/4}, \qquad \gamma=const.>0.
\end{equation}
Of course,  translations \eqref{translation} and dilations \eqref{scaling} do not change the value of $\mathcal{E}$. Incidentally, this shows that any local minimum of the Yamabe functional $\mathcal{E}$ on the sphere or the group has to be a global one.

We turn to the determination of the best constant. Let us define the constants
\begin{equation}  \label{S}
\begin{aligned}
& \Lambda_{\tilde\Theta}\ \overset{def}{=}\ inf\ \left\{ \underset{\QH}{\int} |\nabla v|^2\ Vol_{\tilde\Theta}\ : \ v\in \domG,\ v \geq 0,\ \underset{\QH}{\int} |v|^{2^*}\ Vol_{\tilde\Theta}\ =\ 1 \right\}\\
& \text{and}\\
& \Lambda \overset{def}{=}\ inf\ \left\{ \underset{\QH}{\int} |\nabla v|^2\ dH\ : \ v\in \domG,\ v \geq 0,\ \underset{\QH}{\int} |v|^{2^*}\ dH\ =\ 1 \right\}.
\end{aligned}
\end{equation}
Clearly, $\Lambda_{\tilde\Theta}\ {=}\
S_{\tilde\Theta}^{-2}$, where $S_{\tilde\Theta}$ is the best constant in the $L^2$ Folland-Stein
inequality
\begin{equation} \label{e:FS-theta}
\left(\int_{\QH} \ |u|^{2^*}\ Vol_{\tilde\Theta}\right)^{1/2^*} \leq\ S_{\tilde\Theta}\ \left(\int_{\QH }\ |\nabla^{\tilde\Theta} u|^2\ Vol_{\tilde\Theta}\right)^{1/2},
\end{equation}
while $\Lambda\ {=}\ S_{2}^{-2}$ is the best constant in the $L^2$ Folland-Stein
inequality \eqref{FS} (taken with respect to the Lebesgue measure !).  By Remark \ref{r:volume form vs lebesgue} we have $$\Lambda_{\tilde\Theta}\ {=}\ \left [ 2^{-3}(2n)! \right ]^{1/(2n+3)}\ \Lambda.$$
Furthermore, by Lemma \ref{l:const mnimizer} and equations \eqref{e:invariance of dirichlet 2} and \eqref{e:g def} with $g=const$, we have
\begin{multline*}
\Lambda_{\tilde\Theta}\ =\ \frac {1}{S_2^2}=\frac {\int_{\bG}\ |\nabla^{\tilde\Theta} F|^2\, Vol_{\tilde\Theta}}
{ \left [ \int_{\bG} |F|^{2^*}\, Vol_{\tilde\Theta} \right ]^{2/2^*} }\\
=\ \frac {\int_{S^{4n+3}}\ \left ( |\nabla^{\tilde\eta} g|^2 + \frac {\tilde S}{a}\, g^2 \right ) \,
 Vol_{\tilde\eta}}{\left [ \int_{S^{4n+3}} |g|^{2^*}\, Vol_{\tilde\eta}\right ]^{2/2^*}}\ =\ 4n(n+1)\left [ \left ( (2n)! \right) \omega_{4n+3}\right ]^{1/(2n+3)}.
\end{multline*}
 Here, $$\omega_{4n+3}=2\pi^{2n+2}/\Gamma(2n+2)=2\pi^{2n+2}/(2n+1)!$$ is the volume of the unit sphere $S^{4n+3}\subset \mathbb{R}^{4n+4}$ and we also took into account Remark \ref{r:standctct form} which shows that $ Vol_{\tilde\eta}$ gives  $2^{2n+3}\left ( (2n)! \right) \omega_{4n+3}$ for the volume of $S^{4n+3}$.  Thus,
\[
S_{\tilde\Theta}\ =\ \left ( 4n(n+1)\left [ \left ( (2n)! \right) \omega_{4n+3}\right ]^{1/(2n+3)}\right )^{-1/2}\ =\ \frac { \left [ \left ( (2n)! \right) \omega_{4n+3}\right ]^{-1/(4n+6)}}{2\sqrt{n(n+1)}},
\]
which completes the proof of part a).

b) The Yamabe constant of the sphere is calculated immediately by taking a constant function in \eqref{e:Yamabe energy}  \begin{equation}\label{e:yamabe constant 2}
\lambda(S^{4n+3}, [\tilde\eta])=a\, \Lambda_{\tilde\Theta}, \qquad a=4\frac {Q+2}{Q-2}=4\frac {n+2}{n+1}.
\end{equation}

This completes the proof of Theorem \ref{t:main}.
\end{proof}

\begin{rmrk}
In view of the above Lemmas it follows that in the conformal class of the standard qc structure on the sphere (or the quaternionic Heisenberg group) there is an extremal  qc contact form for problem \eqref{e:Yamabe constant problem} which is also qc-Einstein, see \cite[Definition 4.1]{IMV}, and has partial symmetry, see \cite[Definition 1.2]{GV}, if viewed as a qc structure on the group.
Thus, the above precise constants and extremals can also be taken directly from \cite[Theorem 1.1 and 1.2]{IMV} or the result of \cite{GV}.  However, the  functions \eqref{e:extremals}  depend
on one more arbitrary multiplicative parameter $\gamma$ since in the current paper we are dealing with the functions realizing the infimum of \eqref{e:Yamabe energy} rather than with the qc Yamabe equation with a fixed qc scalar curvature.
\end{rmrk}

\end{document}